\ifdefined\pdfoutput
  \pdfoutput=1
\fi
\documentclass[10pt]{article}
\usepackage[T1]{fontenc}
\usepackage{lmodern}
\usepackage[a4paper, margin=1in]{geometry}
\usepackage{amsmath, amssymb, amsthm}
\usepackage[english]{babel}
\usepackage{microtype}
\usepackage[numbers,sort&compress]{natbib}
\usepackage{parskip}
\usepackage[title]{appendix}
\usepackage{xcolor}
\usepackage{hyperref}
\usepackage[nameinlink]{cleveref}
\definecolor{linkblue}{RGB}{25,74,135}
\definecolor{citegreen}{RGB}{35,100,70}

\theoremstyle{plain}
\newtheorem{theorem}{Theorem}[section]
\newtheorem{proposition}[theorem]{Proposition}
\newtheorem{corollary}[theorem]{Corollary}

\crefname{theorem}{Theorem}{Theorems}
\crefname{corollary}{Corollary}{Corollaries}
\crefname{proposition}{Proposition}{Propositions}
\crefname{section}{\S}{Sections}
\crefname{appendix}{Appendix}{Appendices}
\Crefname{appendix}{Appendix}{Appendices}

\AddToHook{env/theorem/begin}{\crefalias{section}{theorem}}
\AddToHook{env/proposition/begin}{\crefalias{theorem}{proposition}}
\AddToHook{env/corollary/begin}{\crefalias{theorem}{corollary}}

\newcommand{\Z}{\mathbb{Z}}
\newcommand{\Ann}{\operatorname{Ann}}
\newcommand{\nilind}{\operatorname{nilind}}
\newcommand{\ordplus}{\operatorname{ord}_+}
\newcommand{\gr}{\operatorname{gr}}
\newcommand{\Wcal}{\mathcal{W}}
\newcommand{\Rint}[1]{R_{#1}}
\newcommand{\Rintab}[2]{R_{#1,#2}}

\title{Prime-Interval Algebras}
\author{Joseph M. Shunia\thanks{Independent researcher, Ann Arbor, Michigan, United States.
Email: \texttt{jshunia@gmail.com}.}}
\date{July 2026}

\hypersetup{
  colorlinks=true,
  linkcolor=linkblue,
  citecolor=citegreen,
  urlcolor=linkblue,
  pdftitle={Prime-Interval Algebras},
  pdfauthor={Joseph M. Shunia},
  pdfsubject={An algebraic encoding of interval primes by modular exponentiation},
  pdfkeywords={prime intervals, central binomial coefficients, quotient rings,
    Frobenius, truncated polynomials, finite rings}
}

\begin{document}
\maketitle

\begin{abstract} \noindent
Starting from a positive integer $n$ and no a priori information about the primes above it, we construct a polynomial quotient ring that recovers exactly the primes in $(n,2n]$ from a single modular exponentiation. The primes occur simultaneously as the nonzero monomial degrees of the resulting polynomial remainder, and each coefficient independently certifies its corresponding prime through its additive order.
\\[2mm]
When $n=p_k$ is prime, the least nonzero degree is $p_{k+1}$. Thus the next prime is recovered from the preceding prime alone, without using the index $k$, the prime-counting function, a prime table, nor any primality tests. We develop the underlying ring structure, give equivalent annihilator and quotient formulations, extend the result to shorter intervals, and provide a SageMath implementation.
\\[2mm]
\end{abstract}

\section{Introduction}

This paper was motivated by the following two arithmetic questions:

\begin{enumerate}
\item Can one construct, from a single positive integer $n$, one finite algebraic object that reveals every prime in some explicitly bounded range above $n$, such as $(n,2n]$, without first knowing how many such primes there are or what they are?
\item Given only a known prime value $p_k$, can one recover the next prime $p_{k+1}$ without using the index $k$ or a previously generated list of primes?
\end{enumerate}

The answer to both questions is yes. For each positive integer $n$, define
\begin{equation}
\label{eq:dyadic-modulus}
\Rint{n}
:=
\frac{\binom{2n}{n}}
{\gcd\!\left(\binom{2n}{n},(n!)^2\right)}.
\end{equation}
The key arithmetic fact, proved in \cref{prop:central-binomial-modulus}, is the unexpected identity
\[
\Rint{n}
=
\prod_{\substack{n<p\le2n\\p\ \mathrm{prime}}}p.
\]
Set
\[
B_n=(\mathbb Z/\Rint{n}\mathbb Z)[X]/(X^{2n+1}),
\qquad
f_n=(1+X)^{\Rint{n}}-1\in B_n.
\]
The ring $B_n$ is the algebraic object associated with the dyadic interval $(n,2n]$, and $f_n$ is the distinguished element whose monomial support records its primes. The Frobenius endomorphisms of the residue-field factors give the explicit identity
\[
f_n
=
\sum_{\substack{n<p\le2n\\p\ \mathrm{prime}}}
\frac{\Rint{n}}{p}X^p.
\]
Thus the support exponents are precisely the primes in $(n,2n]$, while the coefficient at $X^p$ has additive order $p$. Bertrand's theorem \cite{HardyWright} ensures that this support is nonempty. Its least exponent is the least prime greater than $n$; when $n=p_k$ is already known to be prime, that exponent is $p_{k+1}$ and $k=\pi(p_k)$ is not part of the construction. The same least prime can also be read from an annihilator threshold. The localized quotient
\[
\mathbb Z[1/n!]\big/\bigl((2n)!/n!\bigr)
\cong
\mathbb Z/\Rint{n}\mathbb Z
\]
explains the arithmetic origin of the modulus.

\subsection{Position among prime formulas, sieves, and polynomial criteria}

The arithmetic starting point is classical: every prime in $(n,2n]$ divides $\binom{2n}{n}$ \cite{Pomerance,HardyWright}, and the corresponding interval product is OEIS A261130 \cite{OEISIntervalProduct}. The present construction turns this divisibility into a two-stage recovery. First, the prime-free $\gcd$ quotient extracts the squarefree interval product without identifying its factors. Frobenius then converts that modulus into one finite-ring element whose support degrees and coefficient orders recover all of those factors simultaneously. This passage from divisibility to explicit prime support is the new step.

This places the construction between interval sieves and prime-representing formulas. Classical sieves recover the whole interval by iteratively eliminating composites; here the small-prime contributions are removed in one arithmetic quotient, after which the polynomial remainder displays all surviving primes at once. A computational comparison appears in \cref{sec:sieve-comparison}. Prime-representing functions in the tradition of Mills and Wright instead produce one prime per argument using specially chosen real constants \cite{Mills,Wright,Matomaki}. Earlier work of Prunescu and Shunia gives fixed-length arithmetic terms for $\pi(m)$, $p(m)$, and extremal prime factors \cite{PrunescuShuniaPrimeTerms,PrunescuShuniaDivisors}. In particular, the latter extracts the greatest prime at most $m$ from $\binom{m}{\lfloor m/2\rfloor}$; at $m=2n$, it selects the largest prime in $(n,2n]$. The construction here retains and separates every prime in that interval, so its output is endpoint-driven and set-valued rather than indexed or scalar-valued.

Frobenius provides a further connection with polynomial congruence criteria. Tests such as AKS apply a Frobenius congruence to one specified candidate \cite{AKS}; here the coefficient ring carries all interval residue characteristics simultaneously, and the single element $f_n$ places each surviving characteristic $q$ at exponent $q$. Thus
\[
n
\longmapsto
(B_n,f_n)
\longmapsto
\{r:[X^r]f_n\ne0\}
=
\{q:n<q\le2n,\ q\text{ prime}\}.
\]
When $n=p_k$, the least support degree is $p_{k+1}$; the annihilator formulation developed below gives the same extraction.

\subsection{Organization of the paper}
\Cref{sec:coefficient-ring} proves the product formula and its localized interpretation. \Cref{sec:frobenius} proves the prime-support expansion, and \cref{sec:annihilator} records its principal annihilator and quotient consequences. \Cref{sec:general-intervals,sec:layered-regrading} develop the algebraic extensions. \Cref{sec:computation} discusses direct evaluation and computational costs, followed by worked computations in \cref{sec:worked-examples}. The appendices give the modular factorial evaluation and SageMath code.

\section{The coefficient ring}
\label{sec:coefficient-ring}

We first show that the integer in \eqref{eq:dyadic-modulus} is precisely the squarefree product of the primes in $(n,2n]$.

\begin{proposition}[Central-binomial formula for the dyadic modulus]
\label{prop:central-binomial-modulus}
For every $n\ge1$,
\begin{equation}
\Rint{n}
=
\frac{\binom{2n}{n}}
{\gcd\!\left(\binom{2n}{n},(n!)^2\right)}
=
\prod_{\substack{n<p\le2n\\p\text{ prime}}}p.
\end{equation}
\end{proposition}

\begin{proof}
Write $C=\binom{2n}{n}$. If $p$ is prime and $n<p\le2n$, then $p$ occurs once in $(2n)!$ and not at all in $n!$, so
\[
v_p(C)=1.
\]

Now let $\ell\le n$ be prime. Legendre's formula \cite{HardyWright} gives
\[
v_\ell(C)
=
\sum_{j\ge1}
\left(
\left\lfloor\frac{2n}{\ell^j}\right\rfloor
-
2\left\lfloor\frac{n}{\ell^j}\right\rfloor
\right).
\]
Each summand is either $0$ or $1$. For $\ell^j\le n$, it is at most $\lfloor n/\ell^j\rfloor$, and there is at most one index $j$ with $n<\ell^j\le2n$. Hence
\[
v_\ell(C)\le v_\ell(n!)+1.
\]
For $n\ge2$, one has $v_\ell(n!)\ge1$, and therefore
\[
v_\ell(C)\le2v_\ell(n!).
\]
Thus the $\gcd$ with $(n!)^2$ removes the full primary factor of $C$ at every prime at most $n$, while each prime in $(n,2n]$ survives once. The case $n=1$ is immediate.
\end{proof}

\begin{corollary}[Residue-field decomposition]
The Chinese remainder theorem gives a canonical isomorphism
\[
\mathbb Z/\Rint{n}\mathbb Z
\cong
\prod_{\substack{n<p\le2n\\p\text{ prime}}}\mathbb F_p.
\]
In particular, the coefficient ring is finite, reduced, and Artinian, with cardinality and characteristic both equal to $\Rint{n}$.
\end{corollary}

\begin{proposition}[Localized provenance]
There is a canonical isomorphism
\[
\mathbb Z[1/n!]\big/\bigl((2n)!/n!\bigr)
\cong
\mathbb Z/\Rint{n}\mathbb Z.
\]
\end{proposition}

\begin{proof}
Every composite integer $m$ satisfying $n<m\le2n$ has all prime divisors at most $n$: a prime divisor larger than $n$ would force $m$ to exceed $2n$. These small prime factors become units in $\mathbb Z[1/n!]$. Every prime $p\in(n,2n]$ occurs exactly once in $(2n)!/n!$, because $2p>2n$. Hence
\[
\frac{(2n)!}{n!}=u\Rint{n}
\]
for a unit $u\in\mathbb Z[1/n!]^\times$, so the two elements generate the same ideal after localization. Since $\gcd(n!,\Rint{n})=1$, localizing $\mathbb Z/\Rint{n}\mathbb Z$ at $n!$ changes nothing.
\end{proof}

The direct presentation $\mathbb Z/\Rint{n}\mathbb Z$ drives the construction below, while the localized presentation explains why prime factors at most $n$ disappear and the primes in $(n,2n]$ remain.

\section{Frobenius and monomial support}
\label{sec:frobenius}

Fix $n\ge1$. Throughout this section, write
\[
A=\mathbb Z/\Rint{n}\mathbb Z,
\qquad
B=A[X]/(X^{2n+1}),
\qquad
f=(1+X)^{\Rint{n}}-1\in B.
\]
Every polynomial class in $B$ is represented uniquely by a polynomial of degree at most $2n$; coefficient extraction and support refer to that representative. Integers written as coefficients denote their residue classes in $A$. For an element $a$ of finite additive order, let
\[
\ordplus(a)=\min\{d\ge1:da=0\}.
\]
For a prime $p\in(n,2n]$, put $m_p=\Rint{n}/p$. By \cref{prop:central-binomial-modulus}, $p\nmid m_p$.

\begin{theorem}[Fiberwise Frobenius monomial]
\label{thm:frobenius-monomial}
Under the decomposition
\[
B\cong
\prod_{\substack{n<p\le2n\\p\text{ prime}}}\mathbb F_p[X]/(X^{2n+1}),
\]
the characteristic-$p$ component of $f$ is
\[
\left(\frac{\Rint{n}}{p}\bmod p\right)X^p.
\]
\end{theorem}

\begin{proof}
In characteristic $p$, Frobenius gives $(1+X)^p=1+X^p$. Since $\Rint{n}=pm_p$,
\[
(1+X)^{\Rint{n}}-1
=(1+X^p)^{m_p}-1
=\sum_{r=1}^{m_p}\binom{m_p}{r}X^{pr}.
\]
Because $p>n$, every term with $r\ge2$ has degree at least $2p\ge2n+2$ and therefore vanishes modulo $X^{2n+1}$. The surviving coefficient is $m_p\bmod p$, which is nonzero.
\end{proof}

\begin{theorem}[Prime-support expansion]
\label{thm:prime-support}
In $B$,
\[
f
=
\sum_{\substack{n<p\le2n\\p\text{ prime}}}
\frac{\Rint{n}}{p}X^p.
\]
For each interval prime $p$,
\[
\ordplus\!\left(\frac{\Rint{n}}{p}\right)=p,
\qquad
\left|A\cdot\frac{\Rint{n}}{p}\right|=p,
\qquad
\Ann_{\Z}\!\left(\frac{\Rint{n}}{p}\right)=p\Z.
\]
Consequently,
\[
\{0\le r\le2n:[X^r]f\ne0\}
=
\{p:n<p\le2n,\ p\text{ prime}\},
\]
and $f$ has $\pi(2n)-\pi(n)$ nonzero coefficients.
\end{theorem}

\begin{proof}
The integer $\Rint{n}/p$ is divisible by every interval prime other than $p$ and is nonzero modulo $p$. Its CRT components therefore agree with \cref{thm:frobenius-monomial}. Its additive order modulo $\Rint{n}$ is
\[
\frac{\Rint{n}}
{\gcd\!\left(\Rint{n},\Rint{n}/p\right)}=p.
\]
The ideal and integer-annihilator statements are equivalent componentwise forms of the same fact. The support assertion follows by inspecting the displayed expansion.
\end{proof}

\begin{corollary}[Least prime above the input]
The least support exponent of $f$ is the least prime strictly greater than $n$. If $n=p_k$ is a known prime, then
\[
\min\{r:[X^r]f\ne0\}=p_{k+1}.
\]
\end{corollary}

\begin{proof}
By Bertrand's theorem \cite{HardyWright}, some prime lies in $(n,2n]$. The claim therefore follows from \cref{thm:prime-support}.
\end{proof}

\section{Structural consequences}
\label{sec:annihilator}

Fix $n\ge1$ and, for this section, write
\[
N=2n+1,
\qquad
A=\mathbb Z/\Rint{n}\mathbb Z,
\qquad
B=A[X]/(X^N),
\qquad
f=(1+X)^{\Rint{n}}-1\in B.
\]
The support theorem already identifies the interval primes. The filtration below records the same data intrinsically as a chain of ideals.

For $0\le t\le N$, set
\[
I_t=\Ann_A(X^tf)
=\{\alpha\in A:\alpha X^tf=0\text{ in }B\}.
\]
These ideals form an increasing chain
\[
0=I_0\subseteq I_1\subseteq\cdots\subseteq I_N=A.
\]

\begin{theorem}[Annihilator filtration]
Under $A\cong\prod_{\substack{n<p\le2n\\p\text{ prime}}}\mathbb F_p$, the ideal $I_t$ is
\[
I_t
\cong
\prod_{\substack{n<p\le2n\\p\text{ prime},\ p+t\ge N}}\mathbb F_p,
\]
with zero in the remaining components. Hence, as $A$-modules,
\[
I_t/I_{t-1}
\cong
\begin{cases}
\mathbb F_{N-t},&N-t\text{ is prime and }n<N-t\le2n,\\
0,&\text{otherwise}.
\end{cases}
\]
\end{theorem}

\begin{proof}
In the characteristic-$p$ component, $X^tf$ is a unit multiple of $X^{p+t}$. If $p+t\ge N$, this element is zero, so every scalar annihilates it. If $p+t<N$, it is nonzero, and its scalar annihilator in the field is zero. A new component appears between $t-1$ and $t$ precisely when $p+t=N$.
\end{proof}

\begin{corollary}[Associated graded pieces]
The associated graded $A$-module
\[
\gr_I(A)=\bigoplus_{t=1}^{N}I_t/I_{t-1}
\]
has a nonzero degree-$t$ piece exactly when $N-t$ is prime in $(n,2n]$. Every nonzero graded piece has cardinality $N-t$.
\end{corollary}

If $t_1<\cdots<t_k$ are the degrees for which $I_t/I_{t-1}\ne0$, then
\[
\{p:n<p\le2n,\ p\text{ prime}\}
=\{N-t_k,\ldots,N-t_1\},
\qquad
\pi(2n)-\pi(n)=k.
\]
Since the displayed primes are in decreasing order, the differences
$t_{i+1}-t_i$ are the corresponding prime gaps read in reverse prime order.

\begin{corollary}[Least-prime extraction by annihilation]
Let $q(n)$ be the least prime greater than $n$, and put
\[
\lambda_n=\min\{t\ge0:X^tf=0\text{ in }B\}.
\]
Then
\[
q(n)=2n+1-\lambda_n.
\]
In particular, if $n=p_k$ is a known prime, then $q(n)=p_{k+1}$.
\end{corollary}

\begin{proof}
By Bertrand's theorem, $q(n)\le2n$. A monomial $X^t$ annihilates $f$ exactly when $t+r\ge2n+1$ for every prime $r\in(n,2n]$. The strongest condition comes from the least such prime, namely $q(n)$.
\end{proof}

\subsection{Annihilators, quotients, and Jordan lengths}

For a nilpotent element $x$ of a ring $S$, use the convention
\[
\nilind_S(x)=\min\{d\ge1:x^d=0\}.
\]
The interval primes also occur as annihilator exponents, vector-space dimensions,
nilpotency indices, and Jordan-block lengths.

\begin{proposition}[Full annihilator]
In the characteristic-$p$ component of $B$,
\[
\Ann_{\mathbb F_p[X]/(X^N)}(f)=(X^{N-p}).
\]
Therefore
\[
\Ann_B(f)
\cong
\prod_{\substack{n<p\le2n\\p\text{ prime}}}(X^{N-p}).
\]
\end{proposition}

\begin{proof}
The characteristic-$p$ component of $f$ is a unit multiple of $X^p$. A class $g(X)$ satisfies $g(X)X^p=0$ modulo $X^N$ exactly when it is divisible by $X^{N-p}$.
\end{proof}

\begin{corollary}[Least prime as a nilpotency index]
Let
\[
D=B/\Ann_B(f).
\]
Then
\[
D\cong
\prod_{\substack{n<p\le2n\\p\text{ prime}}}\mathbb F_p[X]/(X^{N-p}),
\]
and
\[
\nilind_D(X)=N-\min\{p:n<p\le2n,\ p\text{ prime}\}.
\]
\end{corollary}

\begin{theorem}[Quotient by the prime-support element]
There is an isomorphism
\[
B/(f)
\cong
\prod_{\substack{n<p\le2n\\p\text{ prime}}}\mathbb F_p[X]/(X^p).
\]
In the factor indexed by $p$, the residue characteristic, residue-field cardinality, $\mathbb F_p$-dimension, and nilpotency index of $X$ are all equal to $p$.
\end{theorem}

\begin{proof}
In characteristic $p$, the ideal generated by the component of $f$ is $(X^p)$. Since $p<N$,
\[
\bigl(\mathbb F_p[X]/(X^N)\bigr)/(X^p)
\cong
\mathbb F_p[X]/(X^p).
\]
The basis $1,X,\ldots,X^{p-1}$ gives the $\mathbb F_p$-dimension, and the powers of $X$ give the nilpotency index.
\end{proof}

\begin{corollary}[Jordan lengths]
Let $T$ denote multiplication by $X$.
\begin{enumerate}
\item On the characteristic-$p$ component of the principal ideal $Bf=(X^p)$, the map $T$ is one nilpotent Jordan block of size $N-p$.
\item On the characteristic-$p$ component of $B/(f)$, the map $T$ is one nilpotent Jordan block of size $p$.
\end{enumerate}
Thus each interval prime is recovered either as a Jordan-block length or as the complement of one; compare \cite{AltafiIarrobinoMacias}.
\end{corollary}

\begin{proof}
The ideal $(X^p)$ has basis $X^p,X^{p+1},\ldots,X^{N-1}$, while $\mathbb F_p[X]/(X^p)$ has basis $1,X,\ldots,X^{p-1}$. Multiplication by $X$ shifts each basis forward until zero.
\end{proof}

\section{Extension to general intervals}
\label{sec:general-intervals}

Let $a,b$ be integers satisfying
\[
1\le a<b\le2a.
\]
Two variables are convenient because the lower and upper endpoints vary independently. When $a=n$ and $b=2n$, the formula below and the dyadic formula \eqref{eq:dyadic-modulus} produce the same prime product, although their $\gcd$ presentations are different.

Put $M=b!/a!$ and define
\begin{equation}
\label{eq:general-modulus}
\Rintab{a}{b}
:=
\frac{M}{\gcd\!\left(M,\,(a!)^M\right)}.
\end{equation}
An empty product of primes is understood to be $1$; correspondingly, $\mathbb Z/\mathbb Z$ and an empty product of field factors are interpreted as the zero ring.

\begin{theorem}[General interval modulus]
One has
\[
\Rintab{a}{b}
=
\prod_{\substack{a<p\le b\\p\text{ prime}}}p,
\]
and therefore
\[
\mathbb Z/\Rintab{a}{b}\mathbb Z
\cong
\prod_{\substack{a<p\le b\\p\text{ prime}}}\mathbb F_p.
\]
Moreover,
\[
\mathbb Z[1/a!]\big/\bigl(b!/a!\bigr)
\cong
\mathbb Z/\Rintab{a}{b}\mathbb Z.
\]
\end{theorem}

\begin{proof}
If $a<m\le b$ is composite, every prime divisor of $m$ is at most $a$: otherwise $m$ would be at least twice a prime larger than $a$, hence greater than $2a\ge b$. Let $q\le a$ be prime. If $q$ divides $M$, then $v_q(a!)\ge1$ and
\[
v_q\!\left((a!)^M\right)=M v_q(a!)\ge M\ge v_q(M).
\]
Thus the $\gcd$ in \eqref{eq:general-modulus} removes the full $q$-primary factor of $M$. If $p\in(a,b]$ is prime, then $p$ occurs once in $M$, and no second multiple occurs because $2p>2a\ge b$. Also $p\nmid a!$, so it survives the $\gcd$ once. This proves the product formula and the Chinese remainder decomposition.

After localization at $a!$, every removed factor is a unit, whereas none of the surviving interval primes is inverted. Thus $b!/a!$ and $\Rintab{a}{b}$ generate the same ideal in $\mathbb Z[1/a!]$, which gives the final isomorphism.
\end{proof}

For the rest of this section, write
\[
N=b+1,
\qquad
A=\mathbb Z/\Rintab{a}{b}\mathbb Z,
\qquad
B=A[X]/(X^N),
\qquad
f=(1+X)^{\Rintab{a}{b}}-1\in B.
\]

\begin{theorem}[General interval prime-support expansion]
In $B$,
\[
f
=
\sum_{\substack{a<p\le b\\p\text{ prime}}}
\frac{\Rintab{a}{b}}{p}X^p.
\]
The coefficient at $X^p$ has additive order $\ordplus(\Rintab{a}{b}/p)=p$ in $A$.
\end{theorem}

\begin{proof}
In characteristic $p$, write $\Rintab{a}{b}=pm_p$. Frobenius gives
\[
(1+X)^{\Rintab{a}{b}}-1
=(1+X^p)^{m_p}-1.
\]
Because $p>a$ and $b\le2a$, one has $2p\ge b+1=N$, so every term after the first vanishes modulo $X^N$. The surviving coefficient is $m_p\bmod p$. The global coefficient $\Rintab{a}{b}/p$ has the same CRT components, and its additive order is
\[
\frac{\Rintab{a}{b}}
{\gcd\!\left(\Rintab{a}{b},\Rintab{a}{b}/p\right)}=p.
\]
\end{proof}

The component formulas for the scalar filtration, the full annihilator, $B/(f)$, and the layered expansion extend after replacing $n,2n,2n+1,\Rint{n}$ by $a,b,b+1,\Rintab{a}{b}$. Statements involving a least interval prime or its nilpotency index require that $(a,b]$ contain a prime. A successor-prime conclusion requires in addition that the least prime greater than $a$ be at most $b$.

\section{A layered regrading of the discarded terms}
\label{sec:layered-regrading}

Fix $n\ge1$, put $N=2n+1$, and write $A=\mathbb Z/\Rint{n}\mathbb Z$. The principal quotient $X^N=0$ discards the higher Frobenius monomials. For bookkeeping, consider
\[
\Wcal=A[X,Y]/(X^N-Y).
\]
Eliminating $Y$ identifies $\Wcal$ with $A[X]$; the point of the presentation is that every element has the unique regraded normal form
\[
\sum_{r=0}^{N-1}g_r(Y)X^r,
\qquad g_r(Y)\in A[Y],
\]
where the power of $Y$ records the quotient in Euclidean division of an exponent by $N$.

\begin{theorem}[Layered Frobenius expansion]
For a prime $p\in(n,2n]$, put $m_p=\Rint{n}/p$. In the characteristic-$p$ component of $\Wcal$,
\[
(1+X)^{\Rint{n}}-1
=
\sum_{k=1}^{m_p}
\binom{m_p}{k}
Y^{\lfloor pk/N\rfloor}X^{pk\bmod N}.
\]
Its component of $Y$-degree zero is
\[
\left(\frac{\Rint{n}}{p}\bmod p\right)X^p.
\]
\end{theorem}

\begin{proof}
In characteristic $p$, Frobenius gives
\[
(1+X)^{\Rint{n}}-1
=\sum_{k=1}^{m_p}\binom{m_p}{k}X^{pk}.
\]
Write $pk=N\lfloor pk/N\rfloor+(pk\bmod N)$ and use $X^N=Y$. Since $p>N/2$, the inequality $pk<N$ holds for positive $k$ exactly when $k=1$.
\end{proof}

For $L\ge1$, put
\[
\Wcal^{(L)}=A[X,Y]/(X^N-Y,Y^L).
\]

\begin{proposition}[Finite layered normal form]
\label{prop:finite-layers}
There is an isomorphism
\[
\Wcal^{(L)}\cong A[X]/(X^{NL}).
\]
Every element has a unique normal form
\[
\sum_{j=0}^{L-1}\sum_{r=0}^{N-1}c_{j,r}Y^jX^r.
\]
For a prime $p\in(n,2n]$ and $m_p=\Rint{n}/p$, the characteristic-$p$ component is
\[
(1+X)^{\Rint{n}}-1
=
\sum_{\substack{1\le k\le m_p\\pk<NL}}
\binom{m_p}{k}
Y^{\lfloor pk/N\rfloor}X^{pk\bmod N}.
\]
\end{proposition}

\begin{proof}
The relation $Y=X^N$ turns $Y^L=0$ into $X^{NL}=0$. The normal form and the truncated expansion follow.
\end{proof}

The choice $L=1$ recovers the principal truncation, and $L=2$ retains exactly the terms with $pk<2N$. In every case the component of $Y$-degree zero is the prime-support expansion of \cref{thm:prime-support}.

\section{Computational realization}
\label{sec:computation}

The structural proof decomposes the coefficient ring into prime fields, but a computer algebra implementation need not know those factors. The construction can be evaluated directly over the composite coefficient ring by exact integer arithmetic and modular polynomial exponentiation. A complete SageMath listing appears in \cref{app:sage-code}; worked outputs appear in \cref{sec:worked-examples}. SageMath is a natural environment here because it provides exact integers, residue rings, polynomial rings, and quotient rings in one coercion system \cite{SageMath,SagePolynomialQuotients}.

\subsection{Direct quotient-ring computation}

For $n\ge1$, put
\[
C=\binom{2n}{n},
\qquad
N=2n+1.
\]

\begin{proposition}[Direct CAS computation]
\label{prop:direct-CAS}
After evaluating $\Rint{n}$ by \cref{prop:central-binomial-modulus}, use binary exponentiation in the quotient ring to compute
\[
f=(1+X)^{\Rint{n}}-1
\quad\text{in}\quad
(\mathbb Z/\Rint{n}\mathbb Z)[X]/(X^{2n+1}).
\]
The primes in $(n,2n]$ are exactly the support exponents of $f$. Neither a factorization of $\Rint{n}$ nor an explicit enumeration of smaller primes is an input to this exponentiation.
\end{proposition}

For displays and comparisons with Boolean prime tables, define in this subsection the \textbf{interval prime-indicator polynomial}
\[
\mathsf S_n(X)
=\sum_{r=n+1}^{2n}\mathbf 1_{\{[X^r]f\ne0\}}X^r
\in\mathbb F_2[X].
\]
This polynomial records only the support of $f$; it is not obtained by reducing the coefficients of $f$ modulo $2$. For $n\ge2$, the element $2$ is a unit in $\mathbb Z/\Rint{n}\mathbb Z$, so no coefficient homomorphism to $\mathbb F_2$ is involved.

Binary exponentiation uses fewer than $2\log_2\Rint{n}+2$ truncated polynomial multiplications. With schoolbook multiplication, the total is $O(N^2\log\Rint{n})$ coefficient operations; faster polynomial multiplication may be substituted directly. This counts coefficient operations; a full bit-complexity analysis must also account for the growing coefficient modulus, coefficient storage, and construction of the modulus.

The factorial residue needed in the $\gcd$ can be computed recursively from central binomial coefficients without forming $n!$ as an independent large integer; see \cref{sec:doubling-factorial}.

\subsection{Comparison with classical prime sieves}
\label{sec:sieve-comparison}

The sieve of Eratosthenes and its segmented implementations mark candidate integers by repeatedly crossing out multiples of primes up to the square root of the upper endpoint. Segmentation makes this practical for long intervals while keeping only a bounded block in memory \cite{BaysHudson}. Linear sieves refine the organization of the marking process \cite{Pritchard}; the Atkin--Bernstein sieve instead enumerates representations by selected binary quadratic forms and then removes square multiples \cite{AtkinBernstein}.

The present construction addresses a related output problem by a different mechanism. The central-binomial coefficient, modular factorial recurrence, and $\gcd$ together compute $\Rint{n}$ without explicitly enumerating its prime factors. Truncated polynomial exponentiation then uses Frobenius to place each surviving characteristic in its numerical monomial slot. The Boolean polynomial $\mathsf S_n$ can be compared with the output table of a segmented sieve, while the residue coefficients retain algebraic information absent from a Boolean table.

For rapid enumeration, classical and segmented sieves benefit from small machine-integer arithmetic and mature cache-aware implementations. The present construction serves a different computational purpose: it realizes interval prime support through exact arithmetic in a composite quotient ring and retains coefficient-level algebraic data. The displayed source performs no explicit prime enumeration and uses no prime generator.

There are two useful computational modes. The direct mode works in the composite coefficient ring and requires no factorization. If $\Rint{n}$ is subsequently factored, the CRT mode computes the field components independently and makes the Frobenius mechanism transparent. Because the factors of $\Rint{n}$ are the desired primes, that decomposition is explanatory and diagnostic rather than part of the direct computation.

\section{Worked computations}
\label{sec:worked-examples}

This section gives boundary checks, a large Boolean support display, a coefficient-level example for the direct procedure of \cref{prop:direct-CAS}, and a two-layer expansion.

\subsection{Boundary and empty-interval checks}

For $n=1$, one has $\Rint{1}=2$ and
\[
(1+X)^2-1=X^2
\quad\text{in}\quad
\mathbb F_2[X]/(X^3),
\]
so the construction includes the endpoint prime $2$. The input need not be prime: for $n=8$, one has $\Rint{8}=143$ and
\[
(1+X)^{143}-1=13X^{11}+11X^{13}
\quad\text{in}\quad
(\mathbb Z/143\mathbb Z)[X]/(X^{17}).
\]

For the general-interval construction of \cref{sec:general-intervals}, the integer $M=10!/8!=90$ for $(8,10]$ is removed completely by the $\gcd$, so $\Rintab{8}{10}=1$ and the coefficient ring is the zero ring. For the endpoint-prime interval $(8,11]$, one instead obtains $\Rintab{8}{11}=11$ and
\[
(1+X)^{11}-1=X^{11}
\quad\text{in}\quad
\mathbb F_{11}[X]/(X^{12}).
\]

\subsection{A large support display: the interval \texorpdfstring{$(100,200]$}{(100,200]}}

For $n=100$, the modulus is
\[
\Rint{100}=3{,}383{,}080{,}509{,}296{,}917{,}481{,}189{,}798{,}760{,}796{,}480{,}670{,}771{,}162{,}183.
\]
The corresponding residue polynomial has twenty-one nonzero coefficients. Replacing each nonzero coefficient by $1$ gives
\begin{align*}
\mathsf S_{100}(X)={}&X^{101}+X^{103}+X^{107}+X^{109}+X^{113}+X^{127}+X^{131}\\
&+X^{137}+X^{139}+X^{149}+X^{151}+X^{157}+X^{163}+X^{167}\\
&+X^{173}+X^{179}+X^{181}+X^{191}+X^{193}+X^{197}+X^{199}
\qquad\text{in }\mathbb F_2[X].
\end{align*}
The exponents are the complete set of primes in $(100,200]$. The passage to $\mathsf S_{100}$ is only a Boolean recoding of support; it is not coefficient reduction modulo $2$.

\subsection{A coefficient-level computation: the interval \texorpdfstring{$(7,14]$}{(7,14]}}

Here
\[
\binom{14}{7}=3432,
\qquad
\gcd\!\left(3432,(7!)^2\right)=24,
\qquad
\Rint{7}=143.
\]
In
\[
B=(\mathbb Z/143\mathbb Z)[X]/(X^{15}),
\]
one obtains
\[
f=(1+X)^{143}-1=13X^{11}+11X^{13}.
\]
Thus the nonzero coefficient positions return $\{11,13\}$. The coefficient slots from $8$ through $14$ are
\[
\begin{array}{c|ccccccc}
r&8&9&10&11&12&13&14\\
\hline
[X^r]f&0&0&0&13&0&11&0\\
{}[X^r]\mathsf S_7&0&0&0&1&0&1&0
\end{array}
\]
and
\[
[X^{11}]f=\frac{\Rint{7}}{11}=13,
\qquad
[X^{13}]f=\frac{\Rint{7}}{13}=11.
\]
Their additive orders recover the corresponding primes:
\[
\ordplus(13\bmod143)=11,
\qquad
\ordplus(11\bmod143)=13.
\]
Because the input $7$ is prime, the least support exponent is the next prime. The annihilator form gives the same answer:
\[
X^4f=0,
\qquad
X^3f\ne0,
\qquad
15-4=11.
\]
The factorization $143=11\cdot13$ explains the decomposition $\mathbb Z/143\mathbb Z\cong\mathbb F_{11}\times\mathbb F_{13}$ after the direct computation; it is not required to produce $f$.

\subsection{A two-layer expansion}

For the interval $(7,14]$, take $N=15$ and impose $Y^2=0$ in the finite layered algebra of \cref{prop:finite-layers}. The two residue-field components are
\[
f_{11}=2X^{11}+YX^7,
\qquad
f_{13}=11X^{13}+3YX^{11}.
\]
Reassembling the components modulo $143$ gives
\[
f=13X^{11}+11X^{13}+Y\bigl(78X^7+55X^{11}\bigr).
\]
The constant layer is the principal prime-support expansion, while the coefficient of $Y$ records the second Frobenius terms discarded by the truncation $X^{15}=0$.

\section{Conclusion}

The construction answers the motivating interval-recovery problem: from the lower endpoint $n$ alone, and with no prior information about the primes above it, it recovers every prime in the dyadic interval $(n,2n]$. The first, arithmetic stage forms
\[
\Rint{n}
:=
\frac{\binom{2n}{n}}
{\gcd\!\left(\binom{2n}{n},(n!)^2\right)}
=
\prod_{\substack{n<p\le2n\\p\ \mathrm{prime}}}p.
\]
The left-hand expression depends only on $n$, while the product identity shows what the $\gcd$ quotient accomplishes: it removes the full primary contribution of every prime at most $n$ and leaves exactly the squarefree product of the primes in $(n,2n]$. The second, algebraic stage uses this same integer as both coefficient modulus and exponent in one modular exponentiation:
\[
f_n
:=
(1+X)^{\Rint{n}}-1
=
\sum_{\substack{n<p\le2n\\p\ \mathrm{prime}}}
\frac{\Rint{n}}{p}X^p
\quad\text{in } (\mathbb Z/\Rint{n}\mathbb Z)[X]/(X^{2n+1}).
\]
Fiberwise Frobenius is the mechanism behind this remainder: every interval prime $p$ appears at monomial degree $p$, while its coefficient $\Rint{n}/p$ has additive order $p$.

The same primes reappear intrinsically throughout the resulting algebra. They are the support degrees of $f_n$, the jumps in its annihilator filtration, the residue characteristics and dimensions of its quotient factors, and the lengths of its nilpotent Jordan blocks. The least support degree is the least prime greater than $n$; when $n=p_k$, it is $p_{k+1}$, recovered without the index $k$, a prime table, or candidate-wise primality tests.

The main significance is the exact algebraic encoding: the construction gives a uniform passage from a prime-free integer expression to simultaneous prime support and local ring data. Its extension to intervals $a<b\le2a$ shows that the dyadic case belongs to a broader factor-of-two phenomenon, while the layered regrading records the Frobenius terms discarded by the principal truncation. In short, one endpoint determines one modulus, one finite algebra, and one distinguished element whose algebraic signatures are exactly the primes in its interval.

\begin{appendices}
\crefalias{section}{appendix}

\section{Computing the factorial residue by doubling}
\label{sec:doubling-factorial}

Formula \eqref{eq:dyadic-modulus} appears to require a separate construction of $n!$. For the $\gcd$, however, only the residue of $n!$ modulo $C=\binom{2n}{n}$ is needed. This residue can be computed recursively from central binomial coefficients without forming $n!$ as an independent large integer.

For $r\ge0$, write $C_r=\binom{2r}{r}$. The \textbf{swinging factorial} is
\[
\operatorname{sf}(m)
=
\begin{cases}
C_{m/2},&m\text{ even},\\
mC_{(m-1)/2},&m\text{ odd}.
\end{cases}
\]
Then
\begin{equation}
m!
=\left(\left\lfloor\frac m2\right\rfloor!\right)^2\operatorname{sf}(m).
\end{equation}
This is the swinging-factorial recurrence \cite{OEISSwing,BrentZimmermann}; it follows by writing $m=2r$ or $m=2r+1$ and using
\[
(2r)!=(r!)^2\binom{2r}{r}.
\]

For a fixed modulus $M$, define recursively
\[
f_M(0)=f_M(1)=1,
\qquad
f_M(m)\equiv f_M(\lfloor m/2\rfloor)^2\operatorname{sf}(m)\pmod M.
\]
Induction gives $f_M(m)\equiv m!\pmod M$. Taking $M=C$ therefore permits the computation
\begin{equation}
\Rint{n}
=
\frac{C}{\gcd\!\left(C,\,f_C(n)^2\bmod C\right)}.
\end{equation}
The recursion has logarithmic depth and reduces after each multiplication. It avoids storing a separately constructed value of $n!$ by using central binomial coefficients at the halved recursion arguments. The SageMath implementation in \cref{app:sage-code} uses this modular form; its practical cost is governed by the multiprecision integer and binomial routines used by the CAS \cite{BrentZimmermann}.

\section{SageMath implementation}
\label{app:sage-code}

The following SageMath program implements the direct construction in the composite quotient ring. The factorial residue used in the $\gcd$ is evaluated by the doubling recurrence of \cref{sec:doubling-factorial}; the program never forms $n!$ as a separate integer. SageMath's quotient-ring syntax and modular polynomial powers are documented in \cite{SageMath,SagePolynomialQuotients}. The reporting functions print the modulus, polynomial remainder, support, and coefficient additive orders directly, without calling a prime generator. The final loop displays these outputs for every $1\le n\le50$.

\begin{verbatim}
def swing_mod(m, modulus):
    """Return sf(m) modulo modulus."""
    if m <= 1:
        return Integers(modulus)(1)

    r = m // 2
    value = Integers(modulus)(binomial(2*r, r))
    if m % 2 == 1:
        value *= m
    return value


def factorial_mod_doubling(m, modulus):
    """Return m! modulo modulus using the doubling recurrence."""
    Zm = Integers(modulus)
    if m <= 1:
        return Zm(1)

    half = factorial_mod_doubling(m // 2, modulus)
    return half^2 * swing_mod(m, modulus)


def prime_interval_modulus(n):
    """Return R_n without forming n! as a separate integer."""
    if n < 1:
        raise ValueError("n must be positive")

    C = ZZ(binomial(2*n, n))
    factorial_residue = ZZ(factorial_mod_doubling(n, C))
    square_residue = (factorial_residue^2) % C
    return C // gcd(C, square_residue)


def coefficient_data(modulus, representative, support):
    """Return (degree, coefficient, additive order) for each support term."""
    data = []
    for degree in support:
        coefficient = ZZ(representative[degree])
        additive_order = modulus // gcd(modulus, coefficient)
        data.append((degree, coefficient, additive_order))
    return data


def prime_support_element(n):
    """Compute R_n, the residue element, its support, and its Boolean recoding."""
    Rn = prime_interval_modulus(n)

    A = Integers(Rn)
    P = PolynomialRing(A, 'x')
    x = P.gen()
    B = P.quotient(x^(2*n + 1), names='X')
    X = B.gen()

    f = (1 + X)^Rn - 1
    representative = f.lift()
    support = [r for r in range(0, 2*n + 1)
               if representative[r] != 0]

    P2 = PolynomialRing(GF(2), 'z')
    z = P2.gen()
    indicator = sum(z^r for r in support)
    data = coefficient_data(Rn, representative, support)

    return Rn, f, support, indicator, data


def print_prime_support(n):
    """Print the complete output of the dyadic construction."""
    Rn, f, support, indicator, data = prime_support_element(n)
    print(f"n = {n}")
    print(f"R_n = {Rn}")
    print(f"f_n = {f}")
    print(f"support = {support}")
    print(f"S_n = {indicator}")
    print("coefficient data (degree, coefficient, additive order):")
    for row in data:
        print(f"  {row}")


def short_interval_modulus(a, b):
    """Return R_{a,b} for 1 <= a < b <= 2*a."""
    if not (1 <= a < b <= 2*a):
        raise ValueError("require 1 <= a < b <= 2*a")

    M = prod(ZZ(r) for r in range(a + 1, b + 1))
    small_part = power_mod(factorial(a) % M, M, M)
    return M // gcd(M, small_part)


def short_interval_support_element(a, b):
    """Compute the support element for a short interval."""
    Rab = short_interval_modulus(a, b)
    if Rab == 1:
        return Rab, 0, [], []

    A = Integers(Rab)
    P = PolynomialRing(A, 'x')
    x = P.gen()
    B = P.quotient(x^(b + 1), names='X')
    X = B.gen()
    f = (1 + X)^Rab - 1
    representative = f.lift()
    support = [r for r in range(0, b + 1)
               if representative[r] != 0]
    data = coefficient_data(Rab, representative, support)

    return Rab, f, support, data


def print_short_interval(a, b):
    """Print the complete output for a short interval."""
    Rab, f, support, data = short_interval_support_element(a, b)
    print(f"(a, b] = ({a}, {b}]")
    print(f"R_(a,b) = {Rab}")
    print(f"f_(a,b) = {f}")
    print(f"support = {support}")
    print("coefficient data (degree, coefficient, additive order):")
    for row in data:
        print(f"  {row}")


for n in range(1, 51):
    print_prime_support(n)
    print()
\end{verbatim}

The listing and all displayed examples were run with SageMath 10.9. The recursive routine has depth $O(\log n)$ and reduces every product modulo $\binom{2n}{n}$. The listing prioritizes transparency. For large-scale experiments, its polynomial arithmetic can be replaced by a lower-level truncated multiplication routine without changing the mathematical algorithm.

\end{appendices}

\end{document}